\documentclass[12pt]{amsart}
\usepackage{amssymb,amsfonts,amsmath,amsopn,amstext,amscd,latexsym,amsthm,enumerate,mathrsfs,bbm,color}

\usepackage[margin=0.85in]{geometry}
\usepackage[colorlinks=true,linkcolor=blue,urlcolor=blue, citecolor=blue]{hyperref}

\usepackage[all,cmtip]{xy}
\usepackage{longtable}
\usepackage{graphicx}
\usepackage{float} 
\usepackage{verbatim}

\newtheorem{theorem}{Theorem}[section]
\newtheorem{thm}[theorem]{Theorem}
\newtheorem{lemma}[theorem]{Lemma}

\newtheorem{lem}[theorem]{Lemma}

\newtheorem{prop}[theorem]{Proposition}
\newtheorem{cor}[theorem]{Corollary}

\newtheorem{prob}[theorem]{Problem}

\theoremstyle{remark}

\numberwithin{equation}{section}

\DeclareMathOperator{\Sz}{Sz}
\DeclareMathOperator{\PSL}{PSL}
\DeclareMathOperator{\PSU}{PSU}

\begin{document}

\title[On finite groups with element order set of width $3$]
{On finite groups with element order set of width $3$}

\author{Eric Min}
\address{Westwood High School, 12400 Mellow Meadow Drive, Austin, TX 78750}
\email{eric.lin.min7@gmail.com}

\author{Chloe So}
\address{BASIS Independent Silicon Valley, 1290 Parkmoor Avenue, San Jose, CA 95126}
\email{sochloekaylee@gmail.com}

\author{Yong Yang}
\address{Department of Mathematics, Texas State University, 601 University Drive, San Marcos, TX 78666, USA.}
\makeatletter
\email{yang@txstate.edu}

\author{Eddy Zhang}
\address{Parkland High School, 2700 N Cedar Crest Boulevard, Allentown, PA 18104, USA.}
\email{eddyzhang10@gmail.com}

\author{Mengtian Zhang}
\address{College of Science, China Three Gorges University, Yichang, Hubei 443002, China.}
\email{zmt1110@hotmail.com}

\author{Cody Zhou}
\address{Clement’s High School 4200 Elkins Rd, Sugar Land, TX 77479}
\email{codyzhou64@gmail.com}

\date{\today}
%\thanks{}



\begin{abstract}
We study finite groups in which every element order is divisible by at most three distinct primes. We determine the possible nonabelian composition factors, prove that a nonabelian simple group with this property has at most $15$ prime divisors in its order, and deduce that $\rho(G)\leq192$ for every finite group $G$ with $\alpha(G)\leq3$.
\end{abstract}

\maketitle

%%%%%%%%%%%%%%%%%%%%%%%%%%%%%%%%%%%%%%%%%%%%%%%%%%%%%%%%%%%%%%%%%%%%%%%%%%%%%%%%%%%%%%%%%%%%%%%%%%%%%%
\section{Introduction}

All groups considered throughout this paper are finite.

For a group $G$, let $\pi(G)$ denote the set of prime divisors of $|G|$. We write
\[
 \alpha(G)=\max\{|\pi(o(g))|:g\in G\}
 \qquad\text{and}\qquad
 \rho(G)=|\pi(G)|.
\]
Thus $\alpha(G)$ measures the largest number of distinct prime divisors occurring in a single element order, whereas $\rho(G)$ counts the prime divisors of $|G|$.

In \cite{Shinew}, Shi asked whether $\rho(G)$ is bounded above by a function of $\alpha(G)$. Answering Shi's question, Zhang
\cite{Zhang} proved that if $G$ is solvable, then $\rho(G)$ is
bounded by a quadratic function of $\alpha(G)$ and that for
arbitrary $G$, $\rho(G)$ is bounded by a super-exponential function
of $\alpha(G)$. The result for solvable groups was improved later by
Keller \cite{Keller4} to a linear bound. More recently, Hung and Yang proved that $\rho(G)<210\alpha(G)^4$ for every finite group $G$ \cite{HungYang}.

Shi proposed  the following question in the Kourovka Notebook \cite[20.87]{Khukhro}:

\begin{prob}
    In finite groups, when the order of every element is divisible by at most two (distinct) primes, what are the nonabelian composition factors?
\end{prob}

Shi's question was answered in ~\cite{Li}. As an application, Shu and Yang \cite{ShuYang} showed that if $G$ is a finite group such that $\alpha(G) \leq 2$, then $\rho(G) \leq 15$. 

We consider the corresponding questions when three prime divisors are allowed:

\begin{prob}
    What are the nonabelian composition factors of a finite group in which every element order is divisible by at most three distinct primes?
\end{prob}

\begin{prob}\label{problem1.1} 
    If $G$ is finite and $\alpha(G)\leq3$, how large can $\rho(G)$ be?
\end{prob}

Our global bound is the following.

\begin{theorem}\label{main-theorem} Let $G$ be a finite group such that $\alpha(G) \leq 3$. Then $\rho(G) \leq 192$.
\end{theorem}

In Section~2 we determine the possible nonabelian composition factors of groups $G$ with $\alpha(G)\leq3$. Section~3 uses this classification to bound $\rho(S)$ for nonabelian simple groups $S$ with $\alpha(S)\leq3$. In Section~4 we pass from simple groups to arbitrary finite groups and prove Theorem~\ref{main-theorem}.

\section{Classification of simple groups with element order set of width $3$}

Let $\omega(G)$ denote the set of element orders of $G$, called the spectrum of $G$. The spectra of finite simple groups have been studied extensively; see, for example, \cite{Deng,Grechkoseeva,Kantor,Shi2,Vasilev}. We shall use in particular Buturlakin's descriptions in \cite{Buturlakin1,Buturlakin2,Buturlakin3,Buturlakin4,Buturlakin5,BrandlShi}, together with the results cited below for the remaining families. 

%To prove Problem 1.1, we traced it back to another problem. And addressed Problem 1.1 further within the context of solving the following problem:

%\begin{prob}
    %What are the nonabelian composition factors of finite groups in which the order of every element is divisible by at most three (distinct) primes?
%\end{prob}

We shall prove the following classification.

\begin{thm}\label{main results} 
Let $G$ be a finite group in which every element order is divisible by at most three distinct primes. Then every nonabelian composition factor of $G$ is isomorphic to one of the following groups.
\begin{enumerate}[$(1)$]
\item ${\rm A}_n$, where $5\leq n\leq18$.
\item ${\rm M}_{11}, {\rm M}_{12}, {\rm M}_{22}, {\rm M}_{23}, {\rm M}_{24}, {\rm J}_1, {\rm J}_2, {\rm J}_3, {\rm J}_4, {\rm Suz}, {\rm HS}, {\rm McL}, {\rm Co}_3, {\rm Co}_2, \\ 
{\rm Co}_1, He, {\rm Fi}_{22}, {\rm Fi}_{23}, {{\rm Fi}_{24}}', {\rm HN}, {\rm Th}, {\rm B}, {\rm M}, {\rm O'N}, {\rm Ly}, {\rm Ru}$.
\item $\PSL_{2}(q)$. If $p=2$, then $q\in B_1$, otherwise $q\in B_2$.\\
$\PSL_{3}(q)$, $q\in B_3$. \\
$\PSL_{4}(3),\PSL_{4}(4),\PSL_{4}(5),\PSL_{4}(7),\PSL_{4}(8),\PSL_{4}(9),\\\PSL_{4}(17),$\\
$\PSL_{5}(2),\PSL_{5}(3),\PSL_{5}(4),\PSL_{5}(5),$\\
$\PSL_{6}(2),\PSL_{6}(3),\PSL_{6}(4),\PSL_{7}(2),\PSL_{7}(3),\PSL_{8}(2).$
\item $\PSU_{3}(q)$, $q\in B_4$.\\
$\PSU_{4}(3), \PSU_{4}(4),\PSU_{4}(5),\PSU_{4}(7),\PSU_{4}(8),\\\PSU_{4}(9),\PSU_{4}(17)$\\
$\PSU_{5}(2), \PSU_{5}(3), \PSU_{5}(4),\PSU_{5}(7),\PSU_{5}(9)$\\
$\PSU_{6}(2), \PSU_{6}(3),\PSU_{7}(2), \PSU_{7}(3),\PSU_{8}(2), \PSU_{9}(2).$
\item ${\rm PSp}_{4}(q)$, $q\in B_5\cup B_{5,2}$.\\
${\rm PSp}_{6}(3),{\rm PSp}_{6}(9),{\rm PSp}_{6}(5),{\rm PSp}_{6}(7),{\rm PSp}_{8}(3),{\rm PSp}_{10}(3)$.
\item $\Omega_7(3),\Omega_7(9),\Omega_7(5),\Omega_7(7),\Omega_9(3),\Omega_{11}(3), \newline \Omega_7(2),\Omega_7(4), \Omega_7(8), \Omega_9(2), \Omega_9(4), \Omega_{11}(2), \Omega_{13}(2), \newline \rm P\Omega^{+}_8(3),\rm P\Omega^{+}_8(5),\rm P\Omega^{+}_8(7),\rm P\Omega^{+}_8(9),\rm P\Omega^{-}_8(3),\\ \rm P\Omega^{+}_{10}(3), \rm P\Omega^{-}_{10}(3), \rm P\Omega^{+}_{12}(3), \rm P\Omega^{-}_{12}(3), \newline \Omega_8^{+}(2), \Omega_8^{-}(2), \Omega_8^{+}(4), \Omega_8^{-}(4),\\ \Omega_{10}^{+}(2), \Omega_{10}^{-}(2), \Omega_{12}^{+}(2), \Omega_{12}^{-}(2), \Omega_{14}^{+}(2)$.
\item $F_4(2),F_4(3),F_4(4)$.
\item $E_6(2),^2E_6(2)$.
\item $G_2(q)$, where $q\in B_6$ and $q>2$.
\item $^3D_4(2),^3D_4(4),^3D_4(8),^3D_4(3)$.
\item $\Sz(q)$, $q\in B_7$. 
\item $^2G_2(q)$, $q\in B_8$.
\item $^2F_4(2)', ^2F_4(q)$, $q\in B_9$.
\end{enumerate}

Let $q$ be a power of a prime $p$. The set $B_i$ for $i=1,2,\dots,9$ is defined as:

    $B_1=\{q=2^n>2\ \big|\ |\pi(q+1)|\leq 3, |\pi(q-1)|\leq 3\}$.
    
    $B_2=\{q>3\ \big|\ |\pi(\frac{q+1}{2})|\leq 3, |\pi(\frac{q-1}{2})|\leq 3\}$ where $p>2$. 
    
    $B_3=\{q\ \big|\ |\pi(\frac{q^2+q+1}{(3,q-1)})|\leq 3, |\pi(q^2-1)|\leq 3, |\pi(\frac{p(q-1)}{(3,q-1)})|\leq 3\}$.
    
    $B_4=\{q>2\ \big|\ |\pi(\frac{q^2-q+1}{(3,q+1)})|\leq 3, |\pi(q^2-1)|\leq 3, |\pi(\frac{p(q+1)}{(3,q+1)})|\leq 3\}$.
    
    \[
    B_5=\left\{q=p^f\ \middle|\ \begin{array}{l}
    p\neq2,\quad \left|\pi\!\left(\frac{q^2-1}{2}\right)\right|\leq3,\quad
    \left|\pi\!\left(\frac{q^2+1}{2}\right)\right|\leq3,\\[2pt]
    |\pi(p(q-1))|\leq3,\quad |\pi(p(q+1))|\leq3
    \end{array}\right\}.
    \]
    
    \[
    B_{5,2}=\left\{q=2^f,\ f\geq2\ \middle|\ \begin{array}{l}
    |\pi(q^2-1)|\leq3,\quad |\pi(q^2+1)|\leq3,\\[2pt]
    |\pi(2(q-1))|\leq3,\quad |\pi(2(q+1))|\leq3
    \end{array}\right\}.
    \]
    
    \[
    B_6=\left\{q=p^f\ \middle|\ \begin{array}{l}
    |\pi(q^2-1)|\leq3,\quad |\pi(q^2-q+1)|\leq3,\quad |\pi(q^2+q+1)|\leq3,\\[2pt]
    |\pi(p(q-1))|\leq3,\quad |\pi(p(q+1))|\leq3\quad
    \text{if }p=2\text{ or }p>5
    \end{array}\right\}.
    \]
    
    \[
    B_7=\left\{q=2^{2n+1},\ n\geq1\ \middle|\ \begin{array}{l}
    |\pi(2^{2n+1}-1)|\leq3,\\[2pt]
    |\pi(2^{2n+1}+2^{n+1}+1)|\leq3,\quad
    |\pi(2^{2n+1}-2^{n+1}+1)|\leq3
    \end{array}\right\}.
    \]

    \[
    B_8=\left\{q=3^{2n+1},\ n\geq1\ \middle|\ \begin{array}{l}
    |\pi(3^{2n+1}+3^{n+1}+1)|\leq3,\quad
    |\pi(3^{2n+1}-3^{n+1}+1)|\leq3,\\[2pt]
    |\pi(3^{2n+1}-1)|\leq3,\quad
    \left|\pi\!\left(\frac{3^{2n+1}+1}{2}\right)\right|\leq3
    \end{array}\right\}.
    \]

    $B_9=\left\{q=2^{2n+1},\ n\geq1\ \middle|\ \begin{array}{l}
    |\pi(q^2-1)|\leq3,\quad |\pi(q^2+1)|\leq3,\quad |\pi(q^2-q+1)|\leq3,\\[2pt]
    |\pi(q^2+2^{3n+2}-2^{n+1}-1)|\leq3,\\
    |\pi(q^2-2^{3n+2}+2^{n+1}-1)|\leq3,\\
    |\pi(q^2+2^{3n+2}+q+2^{n+1}+1)|\leq3,\\
    |\pi(q^2-2^{3n+2}+q-2^{n+1}+1)|\leq3
    \end{array}\right\}.$
\end{thm}

For a positive integer $n$, let $\pi(n)$ denote the set  of distinct primes dividing $n$. Denote by $\pi(G)=\pi(|G|)$ the set of distinct prime divisors of the order of $G$. For natural numbers $n_1,\ldots,n_k$, we denote by $(n_1,\ldots,n_k)$ their greatest common divisor, and by $[n_1,\ldots,n_k]$ their least common multiple. Let $\mathscr{F}$ be the set of Fermat primes and $\mathscr{M}$ be the set of Mersenne primes.

\begin{prop}\label{3.1}$($\cite[Proposition 3.1]{manz/wolf}$)$
Suppose $q^n-1=p^m$ for primes $p$ and $q$ and positive integers $m$ and $n$. Then
\begin{enumerate}[$(i)$]
\item $n=1$, $q\in\mathscr{F}$ and $p=2$;
\item $m=1$, $p\in\mathscr{M}$ and $q=2$; or
\item $n=2$, $m=q=3$ and $p=2$.
\end{enumerate}
\end{prop}

\begin{prop}\label{3.2}$($\cite[Proposition 3.2]{manz/wolf}$)$
Suppose $q^m-1=2^n\cdot3$ for a prime $q$ and positive integers $m$ and $n$. Then
\begin{enumerate}[$(i)$]
\item $m=1$ or
\item $m=2$ and $q\in\{5,7\}$.
\end{enumerate}
\end{prop}

\begin{cor} \label{Corollary 2.2}
If $2^n-1$ is a prime, then $n$ is a prime. If $2^n+1$ is a prime, then $n$ is a power of $2$.
\end{cor}  
\begin{proof}
It is a direct result from Proposition \ref{3.1}.
\end{proof}

\begin{lem} \label{lemma 3} $($\cite[Lemma 2.4]{Li}$)$
Let $p$ is a prime and $q$ is a power of $p$, then either $q+1$ or $q-1$ must have a prime factor $3$ when $p\neq3$.
\end{lem}

\begin{lem} \label{lemma 4} $($\cite[Lemma 2.5]{Li}$)$
If $p$ is a prime number greater than $5$ and $q=p^n$, then $q^2-1$ is divisible by $24$. 
\end{lem}

\begin{lem} \label{lemma 5} $($\cite[Lemma 2.6]{Li}$)$
If $(5,m)=1$ for integer $m>1$, then $5\mid (m^4-1).$
\end{lem}

\begin{lem} \label{lemma 9} 
 Let $p$ be an odd prime and let $q = p^k$ for some positive integer $k$. Then the integers $\frac{q-1}{2} \text{and} \,\frac{q+1}{2}$ are coprime.
\end{lem}

\begin{proof}
Since $p$ is odd, $q = p^k$ is odd, so $\frac{q-1}{2}$ and $\frac{q+1}{2}$ are integers. Let $d = \gcd\left(\frac{q-1}{2}, \frac{q+1}{2}\right)$, then $d$ divides their difference: $\frac{q+1}{2} - \frac{q-1}{2} = 1.$
Thus $d \mid 1$, which forces $d = 1$. Hence $\frac{q-1}{2}$ and $\frac{q+1}{2}$ are coprime.
\end{proof}

\begin{lem} \label{K4} 
The Diophantine equation $$2^m+1=3y^n$$ has no solution, where $m, n, y$ are integers such that $m\geq 2, n\geq 2, y\geq 2$.
\end{lem}

\begin{proof}
    See \cite[Sec.3]{BCM}.
\end{proof}

We use Fermat's little theorem in its standard form: if $r$ is prime and $(a,r)=1$, then $a^{r-1}\equiv1\pmod r$.

\begin{theorem}[Mih\u{a}ilescu \cite{Mihailescu}]\label{Mihăilescu's Theorem}
If $x,y,a,b>1$ and $x^a-y^b=1$, then $(x,a,y,b)=(3,2,2,3)$.
\end{theorem}

\begin{cor}\label{Corollary 11}
Let $q$ be a power of the prime $p$. Then
\[
 |\pi(p(q^2-1))|\leq 3
\]
if and only if
\[
 q\in\{2,3,4,5,7,8,9,17\}.
\]
\end{cor}
\begin{proof}
Since $p\nmid q^2-1$, the hypothesis says that $q^2-1=(q-1)(q+1)$ involves at most two primes distinct from $p$.

Suppose first that $p=2$. The two odd integers $q-1$ and $q+1$ have greatest common divisor $1$. Moreover, one of them is divisible by $3$. Hence, if $|\pi(2(q^2-1))|\leq3$, the factor divisible by $3$ is a power of $3$ and the other factor is a power of one odd prime. If $q=2^t>2$, the equations $2^t-1=3^a$ and $2^t+1=3^a$ are settled by Proposition~\ref{3.1} and Mih\u{a}ilescu's theorem: the first gives $t=2$, while the second gives $t=3$. Thus $q\in\{2,4,8\}$.

Now let $p=3$. Write $q=3^t$. For $t>1$, the coprime consecutive integers
\[
 \frac{q-1}{2},\qquad \frac{q+1}{2}
\]
must, under the hypothesis, be powers of $2$ and of one further odd prime. Since $3\nmid q^2-1$, that odd prime cannot be $3$. If $(q-1)/2$ is a power of $2$, then $3^t-2^{a+1}=1$; Mih\u{a}ilescu's theorem gives $t=2$. If $(q+1)/2$ is a power of $2$, then $2^{a+1}-3^t=1$, and Mih\u{a}ilescu's theorem gives no solution with $t>1$. Hence $q\in\{3,9\}$.

Finally, suppose that $p>3$. By Lemma~\ref{lemma 4}, the primes $2$ and $3$ divide $q^2-1$. Thus the only primes that may divide $p(q^2-1)$ are $p,2,3$. Since
\[
 \left(\frac{q-1}{2},\frac{q+1}{2}\right)=1,
\]
the two consecutive integers $(q-1)/2$ and $(q+1)/2$ are, apart from the possibility that one is $1$, powers of $2$ and $3$ in some order. Proposition~\ref{3.1} and Mih\u{a}ilescu's theorem give only
\[
 (q,p)=(5,5),(7,7),(17,17).
\]
Substitution shows that all eight displayed values of $q$ do satisfy the required inequality.
\end{proof}

\begin{cor} \label{Corollary 12}
Let $q$ be a power of the prime $p$. If $|\pi(p(q^3-1))|\leq3$, then either
\[
 q\in\{2,4,8\},
\]
or $q$ is an odd prime such that
\[
 q-1=2^a\qquad\text{and}\qquad q^2+q+1=r^b
\]
for some positive integers $a,b$ and some odd prime $r$.
\end{cor}
\begin{proof}
Write $q=p^f$. Suppose first that $p=2$. If $f\geq4$ and $f\neq6$, Zsigmondy's theorem gives primitive prime divisors of $2^3-1$, $2^f-1$, and $2^{3f}-1$. These three primes are distinct, since the corresponding multiplicative orders of $2$ are $3,f,$ and $3f$. Hence $2^{3f}-1$ has at least three distinct odd prime divisors. If $f=6$, use instead primitive divisors for the exponents $2,3,$ and $18$. Thus $|\pi(2(2^{3f}-1))|\geq4$ whenever $f\geq4$, and consequently $q\in\{2,4,8\}$.

Now assume that $p$ is odd. Since
\[
 p(q^3-1)=p(q-1)(q^2+q+1),\qquad
 (q-1,q^2+q+1)=(q-1,3),
\]
and $p$ divides neither $q-1$ nor $q^2+q+1$, the two latter factors together involve at most two distinct primes. If $3\mid q-1$, then
\[
 v_3(q^2+q+1)=1.
\]
As $2\mid q-1$ and $q^2+q+1>3$, this would force at least three distinct primes in $(q-1)(q^2+q+1)$, a contradiction. Hence
\[
 (q-1,q^2+q+1)=1.
\]
It follows that
\[
 q-1=2^a,\qquad q^2+q+1=r^b
\]
for some odd prime $r$.

It remains to show that $f=1$. If $f>1$, then $p^f-2^a=1$. Mih\u{a}ilescu's theorem gives the unique possibility $p^f=9$ and $2^a=8$. But for $q=9$ we have
\[
 q^2+q+1=91=7\cdot13,
\]
contrary to the preceding conclusion. Thus $f=1$, and $q$ is an odd prime.
\end{proof}

\begin{cor}\label{Corollary 13}
Let $q$ be a power of the prime $p$. If $|\pi(p(q^4-1))|\leq3$, then $q\in\{2,3\}$.
\end{cor}
\begin{proof}
Since $p(q^2-1)$ divides $p(q^4-1)$, Corollary~\ref{Corollary 11} leaves only
$q\in\{2,3,4,5,7,8,9,17\}$. Factoring $p(q^4-1)$ for these eight values leaves only $q=2$ and $q=3$.
\end{proof}

\begin{cor}\label{Corollary 14}
Let $q$ be a power of the prime $p$. If $|\pi(p(q^6-1))|\leq3$, then $q=2$.
\end{cor}
\begin{proof}
Again $p(q^2-1)$ divides $p(q^6-1)$, so Corollary~\ref{Corollary 11} reduces the problem to the eight values
$q\in\{2,3,4,5,7,8,9,17\}$. Direct factorization gives $|\pi(p(q^6-1))|\leq3$ only for $q=2$.
\end{proof}

\begin{cor}\label{Corollary 15}
Let $q$ be a power of the prime $p$. Then $p(q^{12}-1)$ is divisible by each of
\[
 2,3,5,7,13.
\]
In particular, $|\pi(p(q^{12}-1))|\geq5$.
\end{cor}
\begin{proof}
Let $r\in\{2,3,5,7,13\}$. If $r=p$, then $r\mid p(q^{12}-1)$. If $r\neq p$, then $(q,r)=1$. For $r=2$, if $p=2$ then $2\mid p(q^{12}-1)$, while if $p\neq2$ then $q$ is odd and $2\mid q^{12}-1$. For $r\in\{3,5,7,13\}$, Fermat's little theorem gives $q^{r-1}\equiv1\pmod r$, and $r-1$ divides $12$. Hence $r\mid q^{12}-1$. Thus all five displayed primes divide $p(q^{12}-1)$.
\end{proof}

\begin{lem} \label{PSL} 
Let $G$ be the simple group of type $\PSL_n(q)$ such that the order of every element is divisible by at most three distinct primes, where $n\geq2$ and $q$ is a power of a prime $p$. The group $G$ must be one of the following:
\begin{enumerate}[$(1)$]
\item $\PSL_{2}(q)$. If $p=2$, then $q\in B_1$, otherwise $q\in B_2$.
\item $\PSL_{3}(q)$, $q\in B_3$. 
\item $\PSL_{4}(3),\PSL_{4}(4),\PSL_{4}(5),\PSL_{4}(7),\PSL_{4}(8),\PSL_{4}(9),\\\PSL_{4}(17)$.
\item $\PSL_{5}(2),\PSL_{5}(3),\PSL_{5}(4),\PSL_{5}(5)$.
\item $\PSL_{6}(2),\PSL_{6}(3),\PSL_{6}(4)$.
\item $\PSL_{7}(2),\PSL_{7}(3)$.
\item $\PSL_{8}(2)$.
\end{enumerate}
\end{lem}
\begin{proof}
The spectra of $\PSL_n(q)$ are described in \cite[Corollary~3]{Buturlakin1}.
Let $n=2$, then $\omega(\PSL_2(q))=\{2,\text{all divisors of}\ q\pm1\}$ when $p=2$, and $\omega(\PSL_2(q))=\{p,\text{all divisors of}\ \frac{q\pm1}{2}\}$ when $p>2$. Let $n=3$, then $\{\text{all divisors of}\ \frac{q^2+q+1}{(3,q-1)},q^2-1,\ \text{and}\ \frac{p(q-1)}{(3,q-1)}\}\subseteq\omega(\PSL_3(q))$. These descriptions give exactly the conditions defining $B_1$, $B_2$, and $B_3$. We leave these sets in arithmetic form; determining all parameters that satisfy them is a separate number-theoretic problem.

If $n\geq15$, then $p[q^{12}-1,q^{n-14}-1]\in\omega(\PSL_n(q))$. This order is divisible by $p(q^{12}-1)$, which has at least five distinct prime divisors by Corollary~\ref{Corollary 15}. Hence no group in this range satisfies the hypothesis. 

If $9\le n\le14$, then $p(q^6-1)\in\omega(\PSL_n(q))$. Corollary~\ref{Corollary 14} therefore forces $q=2$. For $q=2$ and $n\ge9$, use the block decomposition $2+3+4+(n-9)$: a $2$-dimensional unipotent block together with irreducible semisimple blocks of dimensions $3$ and $4$ yields an element of order $2\,\operatorname{lcm}(2^3-1,2^4-1)=2\cdot3\cdot5\cdot7$. Thus no group in this range satisfies the condition. 

If $n=8$, we have $\{\frac{p(q^6-1)}{(8,q-1)}, p(q^4-1), \frac{(q^8-1)}{(8,q-1)(q-1)}, p(q+1)(q^3-1),p(q^5-1)\}\subseteq \omega(\mathrm{PSL}_8(q))$.Corollary~\ref{Corollary 14} leaves $q=2$ as the only possible parameter in this branch; hence $\PSL_8(2)$ remains as a candidate.

If $n=6$ or $7$, then $p(q^3-1)\in \omega(\mathrm{PSL}_n(q))$. Corollary~\ref{Corollary 12} leaves $q\in\{2,4,8\}$ or the odd-prime branch described there. In the latter branch, the natural $\PSL_4(q)$ section is subject to the $n=4$ calculation below, which uses only Corollary~\ref{Corollary 11}; hence only $q=3,5,17$ need be considered. A direct check with \cite[Corollary 3]{Buturlakin1} for the resulting finite set $q\in\{2,3,4,5,8,17\}$ leaves $\PSL_{6}(2),\PSL_{6}(3),\PSL_{6}(4),\PSL_{7}(2),\PSL_{7}(3)$ as the possible groups in these two ranks.

If $n=5$, we have $\{ \frac{q^5-1}{(5,q-1)(q-1)}, \frac{(q^3-1)(q+1)}{(5,q-1)}, \frac{p(q^3-1)}{(5,q-1)},p(q^2-1)\}\subseteq \omega(\PSL_5(q))$,  by Corollary~\ref{Corollary 11} and a direct verification, the only remaining possible groups are $\PSL_{5}(2),\PSL_{5}(3),\PSL_{5}(4),\PSL_{5}(5)$.

If $n=4$, then $\{\frac{q^4-1}{(4,q-1)(q-1)}, \frac{q^3-1}{(4,q-1)}, \frac{p(q^2-1)}{(4,q-1)}\} \subseteq \omega(\PSL_4(q)) $. To handle $\frac{p(q^2-1)}{(4,q-1)}$, suppose first that $p=2$, $(4,q-1)=1$, and the claim follows by a direct calculation using Corollary~\ref{Corollary 11}. If $p\neq 2$ and $(4,q-1)=2$, then $p(q^2-1)=p(q+1)(q-1)=2^2\cdot p\cdot \frac{q+1}{2}\cdot \frac{q-1}{2}$, so the prime factor $2$ still occurs. If $(4,q-1)=4$, then $\frac{p(q^2-1)}{(4,q-1)}=p\cdot \frac{q+1}{2}\cdot \frac{q-1}{2}$, and at least one of $\frac{q+1}{2}$ or $\frac{q-1}{2}$ is even, so the prime divisor $2$ again appears. Therefore, by Corollary~\ref{Corollary 11} and the remaining elementary parameter check, the only possible groups are $\PSL_{4}(3),\PSL_{4}(4),\PSL_{4}(5),\PSL_{4}(7),\PSL_{4}(8),\PSL_{4}(9),\\\PSL_{4}(17)$.

\end{proof}

\begin{lem} \label{PSU} 
Let $G\cong\PSU_n(q)$ be simple, where $n\geq2$ and $q$ is a power of a prime $p$. If every element order of $G$ is divisible by at most three distinct primes, then $G$ is one of the following:
\begin{enumerate}[$(1)$]
\item $\PSU_{2}(q)$. If $p=2$, then $q\in B_1$, otherwise $q\in B_2$.
\item $\PSU_{3}(q)$, $q\in B_4$.
\item $\PSU_{4}(3), \PSU_{4}(4),\PSU_{4}(5),\PSU_{4}(7),\PSU_{4}(8),\PSU_{4}(9),\\\PSU_{4}(17)$.
\item $\PSU_{5}(2), \PSU_{5}(3), \PSU_{5}(4),\PSU_{5}(7),\PSU_{5}(9)$.
\item $\PSU_{6}(2), \PSU_{6}(3)$.
\item $\PSU_{7}(2), \PSU_{7}(3)$.
\item $\PSU_{8}(2)$.
\item $\PSU_{9}(2)$.
\end{enumerate}
\end{lem}
\begin{proof}
The spectra of $\PSU_n(q)$ are described in \cite[Corollary~3]{Buturlakin1}.
Let $n=2$. If $p=2$, then $\omega(\mathrm{PSU}_2(q))=\{2,\text{all divisors of}\ q+1,\text{and}\ q-1\}$. If $p>2$, then $\omega(\mathrm{PSU}_2(q))=\{p,\text{all divisors of}\ \frac{q+1}{2},\text{and}\ \frac{q-1}{2}\}$. 
If $n=3$, then
\[
 \left\{\frac{q^2-q+1}{(3,q+1)},q^2-1,\frac{p(q+1)}{(3,q+1)}\right\}
 \subseteq\omega(\PSU_3(q)).
\]
This gives exactly the condition $q\in B_4$; as in the linear case, we leave the resulting arithmetic condition in this form.

If $n\geq15$, then $p[q^{12}-1,q^{n-14}-(-1)^{n-14}]\in\omega(\PSU_n(q))$. This order is divisible by $p(q^{12}-1)$, which has at least five distinct prime divisors by Corollary~\ref{Corollary 15}. Hence no group in this range satisfies the hypothesis. 

If $9\leq n\leq14$, then $p(q^6-1)\in\omega(\PSU_n(q))$. By Corollary~\ref{Corollary 14}, the only possible value is $q=2$. The remaining cases can be settled directly from \cite[Corollary 3]{Buturlakin1}. For $n=10,11,12,13,14$, respectively, the spectrum contains numbers divisible by
\[
510=2\cdot3\cdot5\cdot17,\qquad
1020=2^2\cdot3\cdot5\cdot17,\qquad
510,\qquad
4095=3^2\cdot5\cdot7\cdot13,\qquad
8190=2\cdot3^2\cdot5\cdot7\cdot13.
\]
Thus all these groups are excluded. For $n=9$, the same spectrum criterion leaves $\PSU_9(2)$ as a possible case.

If $n=8$, then $\{\frac{p(q^6-1)}{(8,q+1)}, (q^4-1)(q^2+q+1), \frac{(q^8-1)}{(8,q+1)(q+1)}, p(q^3+1)(q-1),\frac{(q^5+1)(q^2-q+1)}{(8,q+1)},\frac{q^7+1}{(8,q+1)}, (q^5+1)(q-1)\}\subseteq \omega(\PSU_8(q))$. Taking $q=p=2$, the preceding restriction leaves $\PSU_8(2)$ as a possible case.

If $n=7$, we have $p(q^4-1)\in \omega(\PSU_7(q))$. By Corollary~\ref{Corollary 13} and a direct computation using \cite[Corollary 3]{Buturlakin1}, the only remaining possible groups are $\PSU_{7}(2),\PSU_{7}(3)$. 
If $n=5$ or $6$, then $p(q^2-1)\in \omega(\PSU_n(q))$. By Corollary~\ref{Corollary 11} and explicit computations using \cite[Corollary 3]{Buturlakin1}, the only remaining possible groups are $\PSU_{6}(2),\PSU_{6}(3)$ and $\PSU_{5}(2),\PSU_{5}(3),\PSU_{5}(4),\PSU_{5}(7),\PSU_{5}(9)$. 
If $n=4$, then $\{\frac{q^4-1}{(4,q+1)(q+1)}, \frac{q^3+1}{(4,q+1)}, \frac{p(q^2-1)}{(4,q+1)}\}\subseteq \omega(\PSU_4(q))$. We next consider $\frac{p(q^2-1)}{(4,q+1)}$. If $(4,q+1)=2$ or $(4,q+1)=4=2^2$, then necessarily $p\neq 2$, and in this case $2^3\mid p(q^2-1)$. Thus division by $(4,q+1)$ does not remove the prime $2$. By Corollary~\ref{Corollary 11} and direct calculation, the only remaining possible groups are $\PSU_{4}(3), \PSU_{4}(4),\\ \PSU_{4}(5),\PSU_{4}(7),\PSU_{4}(8),\PSU_{4}(9),\PSU_{4}(17)$.

\end{proof}

\begin{lem} \label{PSp_{2n}} 
Let $G\cong{\rm PSp}_{2n}(q)$, where $n\geq2$ and $q$ is a power of an odd prime $p$. If every element order of $G$ is divisible by at most three distinct primes, then $G$ is one of the following:
\begin{enumerate}[$(1)$]
\item ${\rm PSp}_{4}(q)$, $q\in B_5$.
\item ${\rm PSp}_{6}(3),{\rm PSp}_{6}(9),{\rm PSp}_{6}(5),{\rm PSp}_{6}(7),{\rm PSp}_{8}(3),{\rm PSp}_{10}(3)$.
\end{enumerate}
\end{lem}
\begin{proof}
The spectra of ${\rm PSp}_{2n}(q)$ are described in \cite[Corollary~2]{Buturlakin2}. 
If $n\geq12$, then $p(q^{12}-1) \in\omega(G)$. By Corollary~\ref{Corollary 15}, the number $p(q^{12}-1)$ has at least five distinct prime divisors. Hence $G$ contains an element whose order has at least five prime factors.

Assume that $n=2$. By \cite[Corollary~2]{Buturlakin2},
\[
 \left\{\frac{q^2-1}{2},\frac{q^2+1}{2},p(q-1),p(q+1)\right\}
 \subseteq\omega({\rm PSp}_{4}(q)).
\]
Hence the hypothesis forces each of these four integers to have at most three distinct prime divisors, and therefore $q\in B_5$. We do not need to solve these arithmetic conditions explicitly.

If $n=3$, then $p(q^2-1) \in\omega({\rm PSp}_6(q))$. By Corollary~\ref{Corollary 11} and the spectrum criterion in \cite[Corollary~2]{Buturlakin2}, the only remaining possible groups are ${\rm PSp}_{6}(3), {\rm PSp}_{6}(5), {\rm PSp}_{6}(7), {\rm PSp}_{6}(9)$. 

If $4\leq n\leq11$, then, since ${\rm PSp}_6(q)$ embeds naturally in ${\rm PSp}_{2n}(q)$, the element order $p(q^3-1)\in\omega({\rm PSp}_6(q))$ also belongs to $\omega({\rm PSp}_{2n}(q))$. Corollary~\ref{Corollary 12} leaves the three even values $2,4,8$ or the odd-prime branch described there. Since $q$ is odd in the present lemma, and the $n=3$ case above already restricts a possible $\mathrm{PSp}_6(q)$ section to $q\in\{3,5,7,9\}$, only $q=3$ or $5$ can occur in the odd-prime branch. The spectrum criterion in \cite[Corollary~2]{Buturlakin2} then eliminates $q=5$ and all ranks except $n=4,5$ when $q=3$, leaving only ${\rm PSp}_{8}(3)$ and ${\rm PSp}_{10}(3)$ as possible groups.

\end{proof}

\begin{lem}\label{PSp4-even}
Let $G\cong{\rm PSp}_4(q)$, where $q=2^f$ and $f\geq2$. Then every element of $G$ has order divisible by at most three distinct primes if and only if $q\in B_{5,2}$.
\end{lem}

\begin{proof}
In even characteristic, ${\rm PSp}_4(q)\cong\Omega_5(q)$. The spectrum description for the characteristic-$2$ symplectic and odd-dimensional orthogonal groups in \cite{Buturlakin2}, specialized to rank $2$, shows that every element order divides one of
\[
 q^2-1,\qquad q^2+1,\qquad 2(q-1),\qquad 2(q+1),\qquad 4,
\]
and each of the first four integers occurs as an element order. The order $4$ contributes only the prime $2$. Thus every element order has at most three distinct prime divisors if and only if
\[
 |\pi(q^2-1)|,\quad |\pi(q^2+1)|,\quad
 |\pi(2(q-1))|,\quad |\pi(2(q+1))|
\]
are all at most $3$. This is precisely the condition $q\in B_{5,2}$.
\end{proof}

\begin{lem} \label{O_{2n+1}} 
Let $G\cong\Omega_{2n+1}(q)$, where $n\geq3$ and $q$ is a power of an odd prime number $p$. If every element order of $G$ is divisible by at most three distinct primes, then $G$ is one of the following:
\begin{enumerate}[$(1)$]
\item $\Omega_7(3),\Omega_7(9),\Omega_7(5),\Omega_7(7)$.
\item $\Omega_9(3)$.
\item $\Omega_{11}(3)$.
\end{enumerate}
\end{lem}
\begin{proof}
The spectra of $G\cong\Omega_{2n+1}(q)$ are described in \cite[Corollary 6]{Buturlakin2}. 

If $n\geq12$, then $p(q^{12}-1) \in\omega(G)$. By Corollary~\ref{Corollary 15}, the number $p(q^{12}-1)$ has at least five distinct prime divisors. Hence $G$ contains an element whose order has at least five prime factors.

If $n=3$, then $\{\frac{q^3\pm1}{2}, (q^2+1)(q\pm1), \frac{p(q^2\pm 1)}{2},p(q\pm 1)\}\subseteq \omega(\Omega_{7}(q))$. By Corollary~\ref{Corollary 11} and the spectrum criterion in \cite[Corollary 6]{Buturlakin2}, the only remaining possible groups are $\Omega_7(3),\Omega_7(5),\Omega_7(7),\Omega_7(9)$.

If $4\leq n\leq11$, then $p(q^4-1)=p(q^2+1)(q^2-1)\in \omega(\Omega_{2n+1}(q))$. By Corollary~\ref{Corollary 13}, only small residual field values remain. Applying the spectrum criterion in \cite[Corollary 6]{Buturlakin2} leaves only $\Omega_9(3)$ and $\Omega_{11}(3)$ as possible cases.

\end{proof}

\noindent{\bf Remark.} When $n=1$ and  $q$ is a power of an odd prime number $p$, we have $\Omega_3(q)\cong {\rm PSL}_2(q)$ which has already been discussed in Lemma~\ref{PSL}. Suppose now that $p=2$. In this case ${\rm PSp}_{2n}(q)\cong \Omega_{2n+1}(q)$ and the spectra of $G=\Omega_{2n+1}(q)\simeq{\rm Sp}_{2n}(q)$ are given in \cite[Corollary 3]{Buturlakin2}. In particular, we have $2(q^2-1)\in \omega(G)$ when $n\geq 3$. By Corollary~\ref{Corollary 11}, the number $2(q^2-1)$ has at most three distinct prime divisors only when $q=2,4,$ or $8$. If $n\ge9$, then $2(q^8-1)\in \omega(G)$, and the number $2(q^8-1)$ has at least four distinct prime divisors. Therefore, it suffices to consider the cases $n=3,4,5,6,7,8$ with $q=2,4,8$. Applying the complete criterion in \cite[Corollary~3]{Buturlakin2} to these finitely many cases leaves
$\Omega_7(2)$, $\Omega_7(4)$, $\Omega_7(8)$, $\Omega_9(2)$, $\Omega_9(4)$, $\Omega_{11}(2)$, and $\Omega_{13}(2)$ as possible groups. Finally, let $n=2$. Then $\Omega_5(q)\cong{\rm PSp}_4(q)$, so these groups are already covered by the symplectic case. Moreover, $\Omega_5(2)'\cong {\rm PSL}_2(9)$, which is already covered by the linear case. Thus no additional group needs to be listed here.

\begin{lem} \label{PO_{2n}} 
Let $G\cong{\rm P}\Omega^{\epsilon}_{2n}(q)$, where $n\geq4,\epsilon\in\{+,-\}$ and $q$ is a power of an odd prime number $p$. If every element order of $G$ is divisible by at most three distinct primes, then $G$ is one of the following:
\begin{enumerate}[$(1)$]
\item ${\rm P}\Omega^{+}_{8}(3), {\rm P}\Omega^{+}_{8}(5), {\rm P}\Omega^{+}_{8}(7), {\rm P}\Omega^{+}_{8}(9), {\rm P}\Omega^{-}_{8}(3)$.
\item ${\rm P}\Omega^{+}_{10}(3)$, ${\rm P}\Omega^{-}_{10}(3)$.
\item ${\rm P}\Omega^{+}_{12}(3)$, ${\rm P}\Omega^{-}_{12}(3)$.
\end{enumerate}
\end{lem}

\begin{proof}
The spectra of ${\rm P}\Omega^{\epsilon}_{2n}(q)$ are given in \cite[Corollary~9]{Buturlakin2} when $(4,q^n-\epsilon)=4$.
If $(4,q^n-\epsilon)=2$, then the center of $\Omega^{\epsilon}_{2n}(q)$ has exponent $1$, and hence
${\rm P}\Omega^{\epsilon}_{2n}(q)\cong \Omega^{\epsilon}_{2n}(q)$.
In this case, the spectra are described in \cite[Corollary~8]{Buturlakin2}.

If $n\ge14$, then $p(q^{12}-1)\in\omega(G)$.
By Corollary~\ref{Corollary 15}, the number $p(q^{12}-1)$ has at least five distinct prime divisors, so no such group satisfies the condition.

Now consider $8\le n\le13$.
By \cite[Corollary~9]{Buturlakin2}, we have $p(q^6-1)\in\omega(G)$.
Since $p$ is odd, Corollary~\ref{Corollary 14} implies that no such group satisfies the condition.

If $n=7$, then $p(q^2+1)(q^3-1)\in\omega(G)$. Since $p(q^3-1)$ divides this element order, Corollary~\ref{Corollary 12} applies. In its odd-prime branch we have $q-1=2^a$ and $q^2+q+1=r^b$. Let $s$ be an odd prime divisor of $q^2+1$. Then $s\neq r$, because $(q^2+1,q^2+q+1)=1$. Hence the displayed element order is divisible by the four distinct primes $2,q,r,s$. Thus the odd-prime branch is impossible. The remaining even values from Corollary~\ref{Corollary 12} do not occur here because $q$ is odd. Therefore this case is excluded.

If $n=6$, then $p(q^4-1)\in\omega(G)$.
By Corollary~\ref{Corollary 13}, it suffices to consider $q=p=3$.
Inspection of the spectra in \cite[Corollary~9]{Buturlakin2} shows that
the only remaining possible groups in this subcase are ${\rm P}\Omega^{+}_{12}(3)$ and ${\rm P}\Omega^{-}_{12}(3)$.

Finally, consider $n=4,5$. The stabilizer of a nonsingular $1$-space has a section of type $\Omega_{2n-1}(q)$. Hence Lemma~\ref{O_{2n+1}} gives an immediate necessary restriction on $q$.

If $n=5$, the corresponding section is $\Omega_9(q)$, so the preceding lemma leaves only $q=p=3$. Hence the only remaining possible groups in this subcase are ${\rm P}\Omega^{+}_{10}(3)$ and ${\rm P}\Omega^{-}_{10}(3)$.

If $n=4$, the corresponding section is $\Omega_7(q)$, so it remains only to consider $q=3,5,7,9$. Here Corollaries~8 and~9 of \cite{Buturlakin2} apply also for $n=4$ and give a complete description of the spectrum. Substitution shows that
\[
 {\rm P}\Omega^{+}_{8}(3),\ {\rm P}\Omega^{+}_{8}(5),\ {\rm P}\Omega^{+}_{8}(7),\ {\rm P}\Omega^{+}_{8}(9)
\]
have no element order divisible by four distinct primes. For minus type, $q=3$ also satisfies the condition. The other three field values are excluded explicitly: Corollary~8 contains the order
\[
 p\left[q\pm1,\frac{q^2+1}{2}\right]
\]
for ${\rm P}\Omega^-_8(q)=\Omega^-_8(q)$. Taking the indicated sign gives
\[
390=2\cdot3\cdot5\cdot13\quad(q=5),\qquad
1050=2\cdot3\cdot5^2\cdot7\quad(q=7),
\]
and
\[
1230=2\cdot3\cdot5\cdot41\quad(q=9).
\]
Thus the only minus-type possibility is ${\rm P}\Omega^-_8(3)$.
\end{proof}

\noindent{\bf Remark.}  The spectra of simple orthogonal groups of even dimension over fields of characteristic $2$ are given in
\cite[Corollary~4]{Buturlakin2}.
If $n\ge10$, then $2(q^4+1)(q^4-1)\in \omega(\Omega_{2n}(q))$. Since $2(q^4-1)$ divides this element order, Corollary~\ref{Corollary 13} shows that a power $q$ of $2$ can satisfy the hypothesis only when $q=2$. For $q=2$ the displayed order is
$2(2^4+1)(2^4-1)=510=2\cdot3\cdot5\cdot17$, so this last case is also excluded.
If $n\ge6$, then
$2(q^2+1)(q^2-1)=2(q^4-1)\in \omega(\Omega_{2n}(q))$.
By Corollary~\ref{Corollary 13}, a power $q$ of $2$ can satisfy the hypothesis only when $q=2$. Thus, for $6\le n\le9$ it remains only to consider $q=2$.
Applying the spectrum criterion in \cite[Corollary~4]{Buturlakin2} to these residual cases leaves $\Omega_{12}^{+}(2)$ and $\Omega_{12}^{-}(2)$ when $n=6$, and $\Omega_{14}^{+}(2)$ when $n=7$; no case remains for $n=8$ or $n=9$.
Next consider $n=5$. By Corollary~\ref{Corollary 11}, the necessary condition $2(q^2-1)\in\omega(\Omega_{10}^{\epsilon}(q))$ restricts $q$ to $2,4,8$. \cite[Corollary~4]{Buturlakin2} applies here. For $q=4$ both signs contain an element of order
\[
2[4-1,4+1,4^2+1]=510=2\cdot3\cdot5\cdot17,
\]
and $q=8$ is excluded in the same criterion. Thus only $\Omega_{10}^{+}(2)$ and $\Omega_{10}^{-}(2)$ remain possible.

For $n=4$, a nonsingular-point stabilizer has a section isomorphic to $\Omega_7(q)\cong\mathrm{Sp}_6(q)$, so the preceding characteristic-$2$ odd-dimensional analysis restricts $q$ to $2,4,8$. If $q=8$, a maximal torus of $\Omega_8^{+}(8)$ has order $8^4-1=4095=3^2\cdot5\cdot7\cdot13$, while a maximal torus of $\Omega_8^{-}(8)$ corresponding to the signed partition $2+2$ has order $[8^2-1,8^2+1]=4095$. Hence $q=8$ is impossible for either sign. Therefore the only remaining possible groups are
$\Omega_8^{+}(2),\ \Omega_8^{-}(2),\ \Omega_8^{+}(4),\ \Omega_8^{-}(4)$.
Finally, for $n\le3$, the structures of these groups are closely related to those of groups discussed earlier, and several classical isomorphisms occur. For example, ${\rm P}\Omega^{+}_{6}(q)\cong {\rm PSL}_4(q)\quad (p\neq2)$, ${\rm P}\Omega^{-}_{6}(q)\cong {\rm PSU}_4(q)$, and ${\rm P}\Omega^{-}_{4}(q)\cong \Omega^{-}_{4}(q)\cong {\rm PSL}_2(q^2)$. Therefore, these cases need not be discussed separately.

\begin{lem}\label{E_6,^2E_6,E_7,E_8}
Let $G$ be a simple group of type $E_6(q)$, ${}^2E_6(q)$, $E_7(q)$, or $E_8(q)$, where $q$ is a power of a prime $p$.
If the order of every element of $G$ is divisible by at most three distinct primes, then
\[
G\cong E_6(2)\quad\text{or}\quad G\cong {}^2E_6(2).
\]
\end{lem}

\begin{proof}
The spectra of $E_6(q)$ and ${}^2E_6(q)$ are described in \cite{Buturlakin3}. In particular, $p(q^4-1)\in \omega(E_6(q)) \quad\text{and}\quad p(q^4-1)\in \omega({}^2E_6(q))$. By Corollary~\ref{Corollary 13}, it suffices to consider the cases $q=p=2$ and $q=p=3$.

For $q=2$, Theorem~1 of \cite{Buturlakin3} gives a finite set $\nu(G)$ with $\mu(G)\subseteq\nu(G)\subseteq\omega(G)$, where $\mu(G)$ denotes the element orders maximal under divisibility. Substituting $q=2$ in that list shows that every number in $\nu(E_6(2))$ and in $\nu({}^2E_6(2))$ is divisible by at most three distinct primes. Hence every element order in these two groups has at most three distinct prime divisors. Thus $E_6(2)$ and ${}^2E_6(2)$ satisfy the condition.

For $q=3$, the group $E_6(3)$ contains an element of order $1680=2^4\cdot3\cdot5\cdot7$, while ${}^2E_6(3)$ contains an element of order $2184=2^3\cdot3\cdot7\cdot13$. Hence neither group satisfies the condition.

The spectra of $E_7(q)$ and $E_8(q)$ are given in \cite{Buturlakin4} and \cite{Buturlakin5}, respectively. In particular, the same $p(q^4-1)$ obstruction reduces these families to $q=2$ or $3$ by Corollary~\ref{Corollary 13}. For $q=3$, a Levi subgroup of $E_7(3)$ has a section of type $E_6(3)$, which has already been excluded above; similarly, $E_8(3)$ has a section of type $E_7(3)$. Hence it remains only to consider $q=2$. The group $E_7(2)$ contains an element of order $2040=2^3\cdot 3\cdot 5\cdot 17$, which has four distinct prime divisors, and hence does not satisfy the condition.
The group $E_8(2)$ has a section of type $E_7(2)$. Since the order of an element in a quotient divides the order of any lift, the element of order $2040$ in that section yields an element of $E_8(2)$ whose order is divisible by $2,3,5,$ and $17$. Thus $E_8(2)$ also fails the condition.
\end{proof}

\begin{lem} \label{F_4,^3D_4} 
Let $G$ be a finite simple group of type $F_4(q)$ or $^3D_4(q)$, where $q$ is a power of a prime $p$. If every element order of $G$ is divisible by at most three distinct primes, then $G$ is one of the following:
\begin{enumerate}[$(1)$]
\item $F_4(2),F_4(3),F_4(4)$.
\item $^3D_4(2),^3D_4(4),^3D_4(8),^3D_4(3)$.
\end{enumerate}
\end{lem}
\begin{proof}

The spectra of the groups $F_4(q)$ and ${}^3D_4(q)$ are given in \cite[Theorems~3.1 and~3.2]{Grechkoseeva}.
In particular,
$p(q^2+1)(q\pm1),\; p(q^3\pm1)\in \omega(F_4(q))
\quad\text{and}\quad
p(q^3\pm1)\in \omega({}^3D_4(q))$.

By Corollary~\ref{Corollary 12}, the condition $|\pi(p(q^3-1))|\leq3$ leaves $q\in\{2,4,8\}$ or the odd-prime branch described there.

For $F_4(q)$, suppose first that $q$ lies in the odd-prime branch. Then $q-1$ is a power of $2$. If $q>3$, both $q+1$ and $q^2+1$ have odd prime divisors, say $r$ and $s$, and $(q+1,q^2+1)=2$, so $r\neq s$. Hence the element order $q(q^2+1)(q+1)$ is divisible by the four distinct primes $2,q,r,s$. Thus the only possible odd value is $q=3$. Among the even values, $q=8$ is excluded by
\[
2(8^2+1)(8-1)=910,
\]
which has four distinct prime divisors. Therefore only $q=2,3,4$ remain as possible parameters for $F_4(q)$.

For ${}^3D_4(q)$, let $q$ lie in the odd-prime branch and suppose $q>3$. Since $q-1$ is a power of $2$, we have $q\equiv2\pmod3$, and hence $3\mid q+1$. By Zsigmondy's theorem there is a primitive prime divisor $r$ of $q^6-1$; it divides $q^2-q+1$ and satisfies $r\neq3$. Therefore
\[
q(q^3+1)=q(q+1)(q^2-q+1)
\]
is divisible by the four distinct primes $2,3,q,r$, a contradiction. Hence only $q=3$ remains from the odd branch. The even values $q=2,4,8$ remain possible, giving the stated list.
Hence only $q=2,3,4,8$ remain as possible parameters for ${}^3D_4(q)$.

It remains to check that the displayed small groups really satisfy the condition. Theorems~3.1 and~3.2 of \cite{Grechkoseeva} give complete descriptions of their spectra by divisors of finitely many maximal orders. Substituting $q=2,3,4$ for $F_4(q)$, the largest number of distinct prime divisors occurring in any of these maximal orders is $3$; for example, maximal-width orders include $30$ for $F_4(2)$, $84$ for $F_4(3)$, and $315$ for $F_4(4)$. Likewise, for ${}^3D_4(q)$ with $q=2,3,4,8$, every maximal order has at most three distinct prime divisors; maximal-width examples are $78$ for $q=3$, $195$ for $q=4$, and $3591$ for $q=8$ (while for $q=2$ the maximum is already $2$). Thus every group in the displayed list satisfies the condition.
\end{proof}

\begin{lem} \label{G_2(q)}
Let $G$ be a finite simple group of type $G_2(q)$, where $q$ is a power of a prime $p$. Then every element order of $G$ is divisible by at most three distinct primes if and only if $q\in B_6$.
\end{lem}
\begin{proof}
Since $G_2(2)$ is not simple, assume $q=p^f>2$. By \cite[Lemma~1.4]{Vasilev}, the maximal element orders are obtained from
\[
q^2-1,\qquad q^2-q+1,\qquad q^2+q+1,
\]
together with $2(q-1)$ and $2(q+1)$ when $p=2$, and $p(q-1)$ and $p(q+1)$ when $p$ is odd; for $p=3,5$ there is also the order $p^2$. Hence the hypothesis immediately gives the conditions in the definition of $B_6$.

Conversely, suppose that $q\in B_6$. If $p=2$ or $p>5$, the conclusion follows from \cite[Lemma~1.4]{Vasilev}. It remains to consider $p=3$ or $5$. Since $|\pi(q^2-1)|\leq3$ and $(q-1,q+1)=2$, if $p(q-1)$ had four distinct prime divisors, then $q-1$ would have three distinct prime divisors and $q+1$ would have to be a power of $2$. Similarly, if $p(q+1)$ had four distinct prime divisors, then $q-1$ would have to be a power of $2$.

We show that neither possibility causes a problem. Suppose first that $q-1=2^a$. If $f$ is odd and $f>1$, then
\[
 q-1=(p-1)(1+p+\cdots+p^{f-1}),
\]
and the second factor is an odd integer greater than $1$, a contradiction. If $f$ is even, write $q=r^2$, where $r=p^{f/2}$. Since $(r-1)(r+1)$ is a power of $2$, the two consecutive even integers $r-1$ and $r+1$ must be $2$ and $4$. Thus $r=3$ and $q=9$. Hence $q=3,5$, or $9$; in these cases $p(q+1)$ has at most three distinct prime divisors. Now suppose that $q+1$ is a power of $2$. If $f$ is even, then $q\equiv1\pmod8$, so $q+1$ cannot be a power of $2$. If $f>1$ is odd, then
\[
 q+1=(p+1)(p^{f-1}-p^{f-2}+\cdots-p+1),
\]
and the second factor is an odd integer greater than $1$, again impossible. Thus $f=1$, and for $p=3,5$ the conclusion is immediate. The remaining maximal order $p^2$ has only one prime divisor. Therefore every element order of $G$ has at most three distinct prime divisors.
\end{proof}

\begin{lem} \label{^2F_4}
Let $G$ be a finite simple group of type ${}^2F_4(q)$, where $q=2^{2n+1}$ and $n\geq1$. Then every element order of $G$ is divisible by at most three distinct primes if and only if $q\in B_9$.
\end{lem}
\begin{proof}
The spectrum of ${}^2F_4(q)$ is given in \cite[Lemma~3]{Deng}. The seven odd quantities appearing in the definition of $B_9$ occur among the maximal element orders, so $\alpha(G)\leq3$ implies $q\in B_9$.

Conversely, suppose that $q\in B_9$ and put $s=\sqrt{2q}=2^{n+1}$. Besides divisors of the seven integers in the definition of $B_9$, the spectrum contains only the orders
\[
2(q+1),\qquad 4(q-1),\qquad 4(q+s+1),\qquad 4(q-s+1),
\]
and the orders $2,4,8,12,16$. The first two have at most three distinct prime divisors because $|\pi(q^2-1)|\leq3$. For the other two, observe that
\[
(q-1,q+s+1)=(q-1,q-s+1)=1.
\]
Indeed, an odd common divisor would give $q\equiv1$ and $s\equiv\mp2$ modulo that divisor, whereas $s^2=2q$, and hence it would divide $2$. Moreover,
\[
(q-1)(q+s+1)=q^2+qs-s-1,
\]
\[
(q-1)(q-s+1)=q^2-qs+s-1,
\]
and $qs=2^{3n+2}$. These are two of the integers occurring in the definition of $B_9$. Since $q-1>1$ and the two factors in each product are coprime odd integers, each of $q+s+1$ and $q-s+1$ has at most two distinct prime divisors. It follows that $4(q\pm s+1)$ has at most three. Thus every element order of $G$ has at most three distinct prime divisors.
\end{proof}

\noindent {\bf Remark.} The exceptional simple group ${}^2F_4(2)'$ must be handled separately, since ${}^2F_4(2)$ itself is not simple. The \emph{Atlas} gives the element orders
\[
1,2,3,4,5,6,8,10,12,13,16,
\]
so every element order is divisible by at most three distinct primes. \\

\begin{proof}[Proof of Theorem~\ref{main results}]
Let $S$ be a nonabelian composition factor of $G$, say $S=H/K$ with $K\lhd H\leq G$. If $xK\in H/K$, then $o(xK)$ divides $o(x)$. Hence
\[
\alpha(S)\leq \alpha(H)\leq \alpha(G)\leq3.
\]
Thus it is enough to determine the nonabelian finite simple groups $S$ with $\alpha(S)\leq3$. By the classification of finite simple groups, it remains to consider alternating groups, sporadic groups, and groups of Lie type.

First consider the alternating groups ${\rm A}_n$. The permutation with disjoint cycle lengths $2,2,3,5,7$ lies in ${\rm A}_{19}$ and has order $210$. Hence ${\rm A}_n$ contains an element whose order is divisible by four distinct primes for every $n\geq19$. Thus only $5\leq n\leq18$ can occur.

For the sporadic simple groups, Inspection of all conjugacy-class orders in the \emph{Atlas of Finite Groups} \cite{CFSG} shows that every element order in each of the $26$ sporadic simple groups is divisible by at most three distinct primes. Hence all sporadic simple groups occur in part~(2).

It remains to consider groups of Lie type. Lemmas~\ref{PSL}--\ref{^2F_4} give the necessary restrictions for the families treated there. For the Suzuki groups, \cite[Theorem~2]{Shi2} gives
\[
 \omega(\Sz(2^{2n+1}))=\{2,4,\ \text{all divisors of }2^{2n+1}-1,
 2^{2n+1}-2^{n+1}+1,
 2^{2n+1}+2^{n+1}+1\},
\]
so the condition is precisely $2^{2n+1}\in B_7$. Similarly, \cite[Lemma~4]{BrandlShi} gives
\[
 \omega({}^2G_2(3^{2n+1}))=\{1,2,3,6,9,\ \text{all divisors of }
 3^{2n+1}\pm3^{n+1}+1,
 3^{2n+1}-1,
 \tfrac{3^{2n+1}+1}{2}\},
\]
and hence the condition is precisely $3^{2n+1}\in B_8$. For $n=0$, the derived subgroup ${}^2G_2(3)'\cong {\rm PSL}_2(8)$ is already covered by the linear case.

\end{proof}

\section{Prime divisors of simple groups of width $3$}

\begin{thm}\label{simpleresults2}
Let $G$ be a nonabelian simple group such that $\alpha(G)\leq3$. Then $\rho(G)\leq15$.
\end{thm}
\begin{proof}
By Theorem~\ref{main results}, it is enough to consider the groups in the classification obtained in the preceding section.

If $G\cong A_n$, where $5\leq n\leq18$, then
\[
 \pi(G)\subseteq\{2,3,5,7,11,13,17\},
\]
so $\rho(G)\leq7$. For the sporadic groups in Theorem~\ref{main results}(2), the orders in the \emph{Atlas} \cite{CFSG} give $\rho(G)\leq15$; equality occurs for the Monster. For the finite lists in parts (3)--(8), (10), and (13), direct substitution in the standard order formulas gives
\[
\begin{array}{c|c|c}
\text{family} & \text{largest value of }\rho & \text{an extremal group}\\ \hline
\text{linear} & 8 & \PSL_6(4)\\
\text{unitary} & 8 & \PSU_9(2)\\
\text{symplectic} & 8 & {\rm PSp}_{10}(3)\\
\text{orthogonal} & 9 & {\rm P}\Omega^-_{12}(3),\ \Omega^+_{14}(2)\\
\text{exceptional} & 8 & E_6(2),\ {}^2E_6(2)
\end{array}
\]
Thus every group in these finite lists has $\rho(G)\leq9$. It remains only to consider the infinite families.

Suppose first that $G\cong\PSL_2(q)$ (equivalently, $\PSU_2(q)$). Put $d=(2,q-1)$. Then
\[
 \pi(G)=\{p\}\cup\pi\!\left(\frac{q-1}{d}\right)
 \cup\pi\!\left(\frac{q+1}{d}\right).
\]
The definitions of $B_1$ and $B_2$ therefore give $\rho(G)\leq7$.

If $G\cong\PSL_3(q)$ and $q\in B_3$, then
\[
 \pi(G)\subseteq\{p\}\cup\pi(q^2-1)
 \cup\pi\!\left(\frac{q^2+q+1}{(3,q-1)}\right),
\]
and hence $\rho(G)\leq7$. The same argument for $\PSU_3(q)$, using $B_4$, gives $\rho(G)\leq7$.

Let $G\cong\operatorname{PSp}_4(q)$. Since
\[
 |G|=\frac{q^4(q^2-1)(q^4-1)}{(2,q-1)},
\]
we have
\[
 \pi(G)=\{p\}\cup\pi(q^2-1)\cup\pi(q^2+1).
\]
If $q$ is odd, the conditions defining $B_5$ give $\rho(G)\leq7$; if $q$ is even, the same conclusion follows from $B_{5,2}$. The isomorphic orthogonal case $\Omega_5(q)$ has the same bound.

Now let $G\cong G_2(q)$ with $q\in B_6$. Since
\[
 |G|=q^6(q^6-1)(q^2-1)
     =q^6(q^2-1)^2(q^2-q+1)(q^2+q+1),
\]
we obtain
\[
 \pi(G)=\{p\}\cup\pi(q^2-1)\cup\pi(q^2-q+1)\cup\pi(q^2+q+1).
\]
Thus $\rho(G)\leq10$.

Let $G\cong\Sz(q)$, where $q\in B_7$, and put $s=\sqrt{2q}$. Since
\[
 |G|=q^2(q-1)(q^2+1),\qquad
 q^2+1=(q+s+1)(q-s+1),
\]
the definition of $B_7$ gives $\rho(G)\leq10$.

Let $G\cong{}^2G_2(q)$, where $q=3^{2n+1}\in B_8$, and put $s=3^{n+1}$. We have
\[
 q^2-q+1=(q+s+1)(q-s+1)
\]
and
\[
 |G|=q^3(q-1)(q+1)(q^2-q+1).
\]
The four integers occurring in the definition of $B_8$ account for all primes outside the defining prime $3$. Moreover, $q-1$ and $(q+1)/2$ are both even. Hence $\rho(G)\leq12$.

Finally, let $G\cong{}^2F_4(q)$ with $q\in B_9$, and put $s=\sqrt{2q}$. We have
\[
 |G|=q^{12}(q^6+1)(q^4-1)(q^3+1)(q-1),
\]
where
\[
 q^4-1=(q-1)(q+1)(q^2+1),\qquad
 q^3+1=(q+1)(q^2-q+1),
\]
and
\[
 q^6+1=(q^2+1)(q^4-q^2+1).
\]
Since $s^2=2q$,
\[
 q^4-q^2+1=(q^2+qs+q+s+1)(q^2-qs+q-s+1).
\]
Thus every odd prime divisor of $|G|$ divides one of
\[
 q^2-1,\quad q^2+1,\quad q^2-q+1,\quad
 q^2+qs+q+s+1,\quad q^2-qs+q-s+1.
\]
These five integers occur among the seven integers in the definition of $B_9$, and each has at most three distinct prime divisors. Since $q\equiv2\pmod 3$, the prime $3$ divides both $q^2-1$ and $q^2-q+1$. Thus the union of the prime divisors of these five odd integers has size at most $5\cdot3-1=14$. Adding the defining prime $2$, we obtain
\[
 \rho(G)\leq 15.
\]
Therefore $\rho(G)\leq15$ in every case. The bound is attained by the Monster.
\end{proof}

\section{Prime divisors of finite groups of width $3$}

The following lemma gives the bound needed for the nonsolvable sections in the proof of the main theorem.

\begin{lemma}\label{lem:semisimple-rho}
Let $B$ be a direct product of nonabelian simple groups. If $\alpha(B)\leq 3$, then $\rho(B)\leq 15$.
\end{lemma}

\begin{proof}
Write
\[
 B=T_1\times\cdots\times T_t,
\]
where each $T_j$ is nonabelian simple. Choose a subset $I\subseteq\{1,\ldots,t\}$ of smallest possible size such that
\[
 \pi\!\left(\prod_{j\in I}T_j\right)=\pi(B).
\]
Relabel the factors indexed by $I$ as $M_1,\ldots,M_s$, and set
\[
 C=M_1\times\cdots\times M_s.
\]
Then $\pi(C)=\pi(B)$, while every proper subproduct of $C$ misses at least one prime divisor of $|B|$.

For each $i$, minimality gives a prime
\[
 q_i\in \pi(M_i)\setminus\bigcup_{j\ne i}\pi(M_j).
\]
These primes are distinct. Fix $i$ and let $x\in M_i$. For every $j\ne i$, choose an element $x_j\in M_j$ of order $q_j$. Because the factors commute, the resulting element of $C$ has prime-divisor set
\[
 \pi(o(x))\cup\{q_j:j\ne i\}.
\]
These sets are disjoint, so
\[
 |\pi(o(x))|+s-1\leq \alpha(B)\leq 3.
\]
Since $x$ was arbitrary,
\[
 \alpha(M_i)\leq 4-s.
\]
In particular, $s\leq 3$.

If $s=1$, then $\rho(B)=\rho(M_1)\leq 15$ by Theorem~3.1. If $s=2$, then $\alpha(M_1),\alpha(M_2)\leq 2$, so \cite[Theorem~2.3]{ShuYang} gives $\rho(M_1),\rho(M_2)\leq 8$. Since every nonabelian finite simple group has even order, the prime $2$ belongs to both $\pi(M_1)$ and $\pi(M_2)$. Hence
\[
 \rho(B)=\rho(C)\leq 8+8-1=15.
\]
If $s=3$, then $\alpha(M_i)\leq 1$ for $i=1,2,3$, and \cite[Theorem~2.5]{ShuYang} gives $\rho(M_i)\leq 4$. Hence
\[
 \rho(B)=\rho(C)\leq 12.
\]
Thus $\rho(B)\leq 15$ in all cases.
\end{proof}

\begin{theorem}\label{main-theorem-final}
Let $G$ be a finite group such that $\alpha(G)\leq 3$. Then $\rho(G)\leq 192$.
\end{theorem}

\begin{proof}
Let $G$ satisfy $\alpha(G)\leq3$. Hung and Yang, following Moret\'o, construct in the proof of \cite[Theorem~3.3]{HungYang} a series
\[
1=S_0\leq R_1<S_1<\cdots<R_m<S_m\leq R_{m+1}=G
\]
such that $R_{i+1}/S_i$ is the largest normal solvable subgroup of $G/S_i$ and $S_i/R_i$ is the socle of $G/R_i$. In particular, each $S_i/R_i$ is a direct product of nonabelian simple groups. The same construction gives a solvable subgroup $H\leq G$ with $\rho(H)=m$. Since $\alpha(H)\leq3$, the solvable case gives
\[
 m=\rho(H)\leq8
\]
by \cite[p.~648]{KellerSmall}.

If $K/L$ is any section of $G$ and $xL\in K/L$, then $o(xL)$ divides $o(x)$. Hence every section appearing above also has element-order width at most $3$. Therefore Lemma~\ref{lem:semisimple-rho} gives
\[
 \rho(S_i/R_i)\leq 15,
\]
while the solvable case gives
\[
 \rho(R_i/S_{i-1})\leq 8
 \qquad\text{and}\qquad
 \rho(G/S_m)\leq 8.
\]
Finally, every prime divisor of $|G|$ occurs in one of the successive sections of the above series. Hence
\[
\pi(G)\subseteq \pi(G/S_m)\cup
\bigcup_{i=1}^m\bigl(\pi(S_i/R_i)\cup\pi(R_i/S_{i-1})\bigr),
\]
and therefore
\[
\begin{aligned}
\rho(G)
&\leq \rho(G/S_m)+
\sum_{i=1}^m\bigl(\rho(S_i/R_i)+\rho(R_i/S_{i-1})\bigr)\\
&\leq 8+m(15+8)\leq 8+8(15+8)=192.
\end{aligned}
\]
\end{proof}

%\section{Acknowledgement} \label{sec:Acknowledgement}
\section*{Acknowledgements.}
This work was partially supported by a grant from the Simons Foundation (\#918096, to YY).

\section*{Disclosure Statement}
The authors declare that they have no competing interests.

\section*{Data Availability Statement} Data sharing is not applicable to this article, as no data sets were generated or analysed during the current study.


\begin{thebibliography}{19}

\bibitem{BrandlShi} R.\,Brandl and W,\,J.\,Shi, A characterization of finite simple groups with abelian Sylow $2$-subgroups, \textit{Ric.\,Mat.} \textbf{42} (1993), 193-198.

\bibitem{BCM} Y. Bugeaud, Z.\,F. Cao, and  M. Mignotte, On simple $K_4$-groups, \textit{J. Algebra} \textbf{241} (2001), 658-668.

\bibitem{Buturlakin1} A.\,A. Buturlakin, Spectra of finite linear and unitary groups, \textit{Algebra and Logic} \textbf{47} (2008), 91–99.

\bibitem{Buturlakin2}
 A.\,A. Buturlakin, Spectra of finite symplectic and orthogonal groups, \textit{Siberian Adv. Math.},  \textbf{21} (2011), 176–210.

\bibitem{Buturlakin3} A.\,A. Buturlakin, Spectra of finite simple groups $E_6(q)$ and ${}^2 E_6(q)$, \textit{Algebra and Logic} \textbf{52}, (2013) 188-202.

\bibitem{Buturlakin4} A.\,A. Buturlakin, Spectra of finite simple groups $E_7(q)$, \textit{Sib. Math. J.}, \textbf{57} (2016), 769-777.

\bibitem{Buturlakin5}
A.\,A. Buturlakin, Spectra of groups $E_8(q)$, \textit{Algebra and Logic}  \textbf{57} (2018), 1–8.



\bibitem{CFSG} {J.\,H. Conway, R.\,T. Curtis, S.\,P. Norton, R.\,A. Parker, and R.\,A. Wilson}, Atlas of Finite Groups, Oxford Univ. Press, New York, 1985.

\bibitem{Deng} H.\,W. Deng and W.\,J. Shi, The characterization of Ree groups ${}^2 F_4(q)$ by their element
orders, \textit{J. Algebra} \textbf{217} (1999), 180-187.

\bibitem{Grechkoseeva} M.\,A. Grechkoseeva and M.\,A. Zvezdina, On spectra of automorphic extensions of finite
simple groups $F_4(q)$ and ${}^3 D_4(q)$, \textit{J. Alg. Appl.} \textbf{15} (2016), 1650168.

\bibitem{HungYang} N.\,N. Hung and Y. Yang, On the prime divisors of element orders, \textit{Math. Nachr.} \textbf{294} (2021), 1905-1911.

\bibitem{Kantor} W.\,M. Kantor and \'{A}. Seress, Large element orders and the characteristic of Lie-type simple groups, \textit{J. Algebra} \textbf{322} (2022), 802-832.

\bibitem{KellerSmall}
T.\,M. Keller, Solvable groups with a small number of prime divisors in the element orders, \textit{J. Algebra} \textbf{170} (1994), 625--648.

\bibitem{Keller4}
T.\,M. Keller, A linear bound for $\rho(n)$, \textit{J. Algebra}
\textbf{178} (1995), 643-652.
  
\bibitem{Khukhro} E.\,I. Khukrho and V.\,D. Mazurov, Unsolved Problems in Group Theory, The Kourovka Notebook, No. 20, Russian Academy of Sciences, 2022.

\bibitem{Li} J.B. Li, X. Xu, Y.C. Xu and Yong Yang, On the width of element order set of a finite group, \textit{Comm. Algebra}, \textbf{53} (2024), 2204–2212.

\bibitem{manz/wolf} {O. Manz and T.\,R. Wolf}, Representations of Solvable Groups, London Mathematical Society, Lecture Notes Series 185, Cambridge University Press, Cambridge, 1993.

\bibitem{Shi2} W.\,J. Shi, A characterization of Suzuki's simple groups, \textit{Proc. Amer. Math. Soc.} \textbf{114} (1992), 589-591.

\bibitem{Shinew}
{W. Shi}, Characterization of finite simple groups and related
topics, \textit{Adv. in Math. (China)} \textbf{20} (1991), 135-141.

\bibitem{ShuYang} {Y. Shu and Y. Yang}, On finite groups with element order set of width 2, \textit{Quaest. Math.}, \url{https://doi.org/10.2989/16073606.2026.2616688}.

\bibitem{Mihailescu} P. Mih\u{a}ilescu, Primary cyclotomic units and a proof of Catalan's conjecture, \textit{J. Reine Angew. Math.} \textbf{572} (2004), 167--195.

\bibitem{Vasilev} A.\,V. Vasil’ev and A.\,M. Staroletov, Recognizability of groups $G_2(q)$ by spectrum, \textit{Algebra
and Logic} \textbf{52} (2013), 1-14.

\bibitem{Zhang} J.\,P. Zhang, Arithmetical conditions on element orders and group structure, 
\textit{Proc. Amer. Math. Soc.} \textbf{123} (1995), 39-44.

\end{thebibliography}
\end{document}